\documentclass{article}%
\usepackage{amsmath}
\usepackage{amsfonts}
\usepackage{amssymb}
\usepackage{graphicx}
\usepackage[bookmarks=true,colorlinks=true, pdfstartview=FitV, linkcolor=blue,
citecolor=blue, urlcolor=blue]{hyperref}
\usepackage{graphicx,color}
\usepackage[compress]{cite}%
\providecommand{\U}[1]{\protect\rule{.1in}{.1in}}
\newtheorem{theorem}{Theorem}

\newtheorem{corollary}{Corollary}

\newtheorem{lemma}{Lemma}

\newtheorem{proposition}{Proposition}
\newtheorem{remark}{Remark}

\newenvironment{proof}[1][Proof]{\noindent\textbf{#1.} }{\ \rule{0.5em}{0.5em}}
\allowdisplaybreaks
\begin{document}

\title{Polynomials with r-Lah coefficient and hyperharmonic numbers}
\author{Levent Karg\i n\thanks{leventkargin48@gmail.com} and M\"{u}m\"{u}n
Can\thanks{mcan@akdeniz.edu.tr}\\Department of Mathematics, Akdeniz University, Antalya Turkey}
\date{}
\maketitle

\begin{abstract}
In this paper, we take advantage of the Mellin type derivative to produce some
new families of polynomials whose coefficients involve $r$-Lah numbers. One of
these polynomials leads to rediscover many of the identities of $r$-Lah
numbers. We show that some of these polynomials and hyperharmonic numbers are
closely related. Taking into account of these connections, we reach several
identities for harmonic and hyperharmonic numbers.



\end{abstract}

\section{Introduction}

Operator theory is one of the attractive way used in number theory as in other
branches of mathematics. One of the often used operators is the operator
$\left(  xD\right)  =x\frac{d}{dx}$, called Mellin derivative \cite{B2}. For a
$n$-times differentiable function $f,$ we have \cite{B3}
\begin{equation}
\left(  xD\right)  ^{n}f\left(  x\right)  =\sum_{k=0}^{n}%
\genfrac{\{}{\}}{0pt}{}{n}{k}%
x^{k}f^{\left(  k\right)  }\left(  x\right)  , \label{02}%
\end{equation}
where $%
\genfrac{\{}{\}}{0pt}{}{n}{k}
$ are the Stirling numbers of the second kind. This operator has a long
mathematical history and firstly used by Euler as a tool in one of his study
\cite{E}. Depending on the choice of function $f$ in (\ref{02}) appear
exponential, geometric, higher-order geometric or harmonic geometric
polynomials \cite{B3,Dil}. These polynomials are closely related to Bernoulli
and Euler numbers and some of their generalizations \cite{B4, Kargin2,
Kargin1, Keller}, which numbers play important roles in number theory.
Further, the operator $\left(  xD\right)  $ has been used in the evaluation of
some power series and integrals \cite{B3,BandDil,Dil,Kargin1,Knopf,MandR}.
Numerous generalizations of the Mellin derivative have also been studied
according to the generalizations of the Stirling numbers of the second kind
\cite{B3,Dil2, Kargin, Kargin1}.

The $r$-Lah numbers $%
\genfrac{\lfloor}{\rfloor}{0pt}{}{n}{k}
_{r}$can be defined by following generating function \cite{NyRa}%
\begin{equation}
\frac{1}{k!}\left(  \frac{t}{1-t}\right)  ^{k}\left(  \frac{1}{1-t}\right)
^{2r}=\sum_{n=k}^{\infty}
\genfrac{\lfloor}{\rfloor}{0pt}{}{n}{k}%
_{r}\frac{t^{n}}{n!}, \label{rLgf}%
\end{equation}
which are also mentioned and studied with different names in \cite{Chen,
Belbachir, MR}. The Lah numbers $%
\genfrac{\lfloor}{\rfloor}{0pt}{}{n}{k}
=%
\genfrac{\lfloor}{\rfloor}{0pt}{}{n}{k}
_{0}$ (rarely called Stirling numbers of the third kind \cite{S}) have many
interesting applications in analysis and combinatorics \cite{Ahuja, Barry,
B5,Guo, GandQ1}.

The hyperharmonic numbers $h_{n}^{\left(  r\right)  }$ are defined by
\cite{CG}
\[
h_{n}^{\left(  r\right)  }=\sum_{k=1}^{n}h_{k}^{\left(  r-1\right)  },\text{
}h_{k}^{\left(  0\right)  }=\frac{1}{k}%
\]
and have the generating function \cite{Benjamin}%
\begin{equation}
\sum_{n=0}^{\infty}h_{n}^{\left(  r\right)  }t^{n}=\frac{-\ln\left(
1-t\right)  }{\left(  1-t\right)  ^{r}}. \label{gfhh}%
\end{equation}
The hyperharmonic numbers are related to the harmonic numbers $H_{n}$\ by
$h_{n}^{\left(  1\right)  }=H_{n}$ and
\begin{equation}
h_{n}^{\left(  r\right)  }=\binom{n+r-1}{n}\left(  H_{n+r-1}-H_{r-1}\right)
\label{hh-h}%
\end{equation}
(cf. \cite{Benjamin,CG}). The harmonic numbers satisfy the \textit{binomial
harmonic identity }\cite{B6}\textit{ }%
\[
H_{n}=\sum_{k=1}^{n}\binom{n}{k}\frac{\left(  -1\right)  ^{k+1}}{k}%
\]
and \textit{symmetric formula}
\begin{equation}
\frac{H_{n}}{n+1}=\sum_{k=1}^{n}\left(  -1\right)  ^{k+1}\binom{n}{k}%
\frac{H_{k}}{k+1}. \label{hsf}%
\end{equation}
The binomial harmonic identity was generalized to binomial hyperharmonic
identity \cite[Corollary 3.1]{Mezo-Dil}
\begin{equation}
h_{n}^{\left(  r\right)  }=\sum_{k=0}^{n}\binom{n}{k}\alpha\left(  k,r\right)
, \label{bhh}%
\end{equation}
with the help of Euler-Siedel matrix method. Here%
\[
\alpha\left(  k,r\right)  =\left\{
\begin{array}
[c]{cc}%
h_{k}^{\left(  r-k\right)  }, & 0\leq k<r\\
\left(  -1\right)  ^{k+\delta_{r}}\left(  r-1\right)  !/k^{\underline{r}}, &
k\geq r
\end{array}
\right.
\]
$\delta_{r}=0$ or $1,$ according to $r$ is even or odd, and $\left(  x\right)
^{\underline{n}}$ is the falling factorial function (see Section 2).

In this paper, we capitalize the operator
\[
\left(  xD+2r\right)  ^{\overline{n}}=\left(  xD+2r\right)  \left(
xD+2r+1\right)  \cdots\left(  xD+2r+n-1\right)  ,
\]
which is the key tool of this study. When it is applied to an appropriate
function $f,$ the coefficients of the resulting function are the $r$-Lah
numbers, namely,\textbf{ }%
\begin{equation}
\left(  xD+2r\right)  ^{\overline{n}}f\left(  x\right)  =\sum_{k=0}^{n}%
\genfrac{\lfloor}{\rfloor}{0pt}{}{n}{k}%
_{r}x^{k}f^{\left(  k\right)  }\left(  x\right)  .\label{3}%
\end{equation}
Some families of polynomials appear due to the choice of $f$ in (\ref{3}). One
of the these polynomials is the exponential $r$-Lah polynomials (see Section
3). With the help of these polynomials, we rediscover many of the identities
given by Nyul and R\'{a}cz \cite{NyRa} for $r$-Lah numbers. In addition, some
new relations are presented for these numbers. The another arising polynomial
is the geometric $r$-Lah polynomial (see Section 4). We reach a relationship
between geometric $r$-Lah polynomials and hyperharmonic numbers. This
connection gives rise to a new formulation for the binomial hyperharmonic
identity (\ref{bhh}):%
\begin{equation}
h_{n+1}^{\left(  r\right)  }=\sum_{k=0}^{n}\binom{n+r}{k+r}\frac{\left(
-1\right)  ^{k}}{k+1}\nonumber\label{SIL1}%
\end{equation}
(see Theorem \ref{bhh-}). Additionally, we show that Bernoulli polynomials can
be described as a sum of the products of hyperharmonic numbers and
$r$-Stirling numbers of the second kind (see Theorem \ref{17-}). This formula
leads a generalization of some identities, obtained by Pascal matrix method,
presented in \cite{CaDa}. Finally, we examine a polynomial which we call
harmonic geometric $r$-Lah polynomial (see Section 5). These polynomials are
also related to hyperharmonic numbers. Considering this, we infer a generating
function for the hyperharmonic numbers with respect to the upper index. We
then deduce a closed-form evaluation formula for the Euler-type sum
\[
\sum_{m=0}^{\infty}\frac{h_{n}^{\left(  m+1\right)  }}{\left(  m+1\right)
^{q}}x^{m}%
\]
(see Theorem \ref{upi}). It should be noted that the sum in question is over
the upper index, however, the studies on the Euler-type sums containing
hyperharmonic numbers depend on the lower index \cite{DM,DB,Kamano}.
Furthermore, we produce some generalizations of the symmetric formula
(\ref{hsf}) (see Theorems \ref{27-} and \ref{29-}), for instance, one of the
generalizations is
\[
\frac{h_{n}^{\left(  r\right)  }}{\left(  n+1\right)  ^{\overline{r}}}%
=\sum_{k=1}^{n}\left(  -1\right)  ^{k+1}\binom{n}{k}\frac{H_{k}}{\left(
k+1\right)  ^{\overline{r}}},
\]
where $\left(  x\right)  ^{\overline{n}}$ is the rising factorial function
(see Section 2). Moreover, we give formulas for the sum of the products of
hyperharmonic numbers (see Theorem \ref{24-}). Particular cases give rise to
some interesting results, for instance,%
\[
2\sum_{k=1}^{n}\frac{H_{k}}{k+1}=\sum_{k=1}^{n}\frac{H_{k}}{n+1-k}=\left(
H_{n+1}\right)  ^{2}-H_{n+1}^{\left(  2\right)  }%
\]
where $H_{n}^{\left(  2\right)  }=1+\frac{1}{2^{2}}+\cdots+\frac{1}{n^{2}}.$

\section{Preliminaries}

Let $\left(  x\right)  ^{\overline{n}}$ and $\left(  x\right)  ^{\underline{n}%
}$ denote the rising and falling factorial functions defined by
\begin{align*}
\left(  x\right)  ^{\overline{n}}  &  =x\left(  x+1\right)  \cdots\left(
x+n-1\right)  ,\text{ with }\left(  x\right)  ^{\overline{0}}=1,\\
\left(  x\right)  ^{\underline{n}}  &  =x\left(  x-1\right)  \cdots\left(
x-n+1\right)  ,\text{ with }\left(  x\right)  ^{\underline{0}}=1.
\end{align*}
The $r$-Stirling numbers of the first kind $%
\genfrac{[}{]}{0pt}{}{n}{k}
_{r}$, the second kind $%
\genfrac{\{}{\}}{0pt}{}{n}{k}
_{r}$ and the $r$-Lah numbers $%
\genfrac{\lfloor}{\rfloor}{0pt}{}{n}{k}
_{r}$can be defined by \cite{Broder, NyRa}%
\begin{align}
\left(  x+r\right)  ^{\overline{n}}  &  =\sum_{k=0}^{n}
\genfrac{[}{]}{0pt}{}{n}{k}
_{r}x^{k}\text{,}\label{s1}\\
\left(  x+r\right)  ^{n}  &  =\sum_{k=0}^{n}
\genfrac{\{}{\}}{0pt}{}{n}{k}
_{r}\left(  x\right)  ^{\underline{k}}\text{,} \label{s2}%
\end{align}
and
\begin{equation}
\left(  x+2r\right)  ^{\overline{n}}=\sum_{k=0}^{n}
\genfrac{\lfloor}{\rfloor}{0pt}{}{n}{k}
_{r}\left(  x\right)  ^{\underline{k}}. \label{s3}%
\end{equation}
Note that $%
\genfrac{[}{]}{0pt}{}{n}{k}
=
\genfrac{[}{]}{0pt}{}{n}{k}
_{0}$ and $%
\genfrac{\{}{\}}{0pt}{}{n}{k}
=
\genfrac{\{}{\}}{0pt}{}{n}{k}
_{0}$ are the Stirling numbers of the first and second kind.

Replace $x$ by $\left(  xD\right)  $ in%
\[
\left(  x\right)  ^{\underline{n}}=\sum_{k=0}^{n}\left(  -1\right)  ^{n-k}
\genfrac{[}{]}{0pt}{}{n}{k}
x^{k}
\]
and then apply it\ to a $n$-times differentiable function $f.$ Utilizing
(\ref{02}) and the identity
\[
\sum_{j=k}^{n}\left(  -1\right)  ^{j-k}
\genfrac{[}{]}{0pt}{}{n}{j}
\genfrac{\{}{\}}{0pt}{}{j}{k}
=\left\{
\begin{array}
[c]{cc}%
1, & n=k\\
0, & n\neq k
\end{array}
\right.
\]
we obtain that
\begin{equation}
\left(  xD\right)  ^{\underline{k}}f\left(  x\right)  =x^{k}f^{\left(
k\right)  }\left(  x\right)  . \label{01}%
\end{equation}
Therefore, applying the operator $\left(  xD+2r\right)  ^{\overline{n}}$ to a
$n$-times differentiable function $f$ and using (\ref{s3}) and (\ref{01}), we
obtain (\ref{3}), i.e.,
\[
\left(  xD+2r\right)  ^{\overline{n}}f\left(  x\right)  =\sum_{k=0}^{n}
\genfrac{\lfloor}{\rfloor}{0pt}{}{n}{k}
_{r}x^{k}f^{\left(  k\right)  }\left(  x\right)  .
\]
By the similar fashion, it follows from (\ref{s2}) and (\ref{01}) that
\begin{equation}
\left(  xD+r\right)  ^{n}f\left(  x\right)  =\sum_{k=0}^{n}
\genfrac{\{}{\}}{0pt}{}{n}{k}
_{r}x^{k}f^{\left(  k\right)  }\left(  x\right)  . \label{2}%
\end{equation}

We finally want to recall the $r$-Stirling transform which will be useful in
the next sections:%
\begin{align*}
a_{n}  &  =\sum_{k=0}^{n}
\genfrac{\{}{\}}{0pt}{}{n}{k}
_{r}b_{k}\text{ }\left(  n\geq0\right)  \text{ if and only if }b_{n}%
=\sum_{k=0}^{n}\left(  -1\right)  ^{n-k}
\genfrac{[}{]}{0pt}{}{n}{k}
_{r}a_{k}\text{ }\left(  n\geq0\right)  ,\\
a_{n}  &  =\sum_{k=0}^{n}
\genfrac{[}{]}{0pt}{}{n}{k}
_{r}b_{k}\text{ }\left(  n\geq0\right)  \text{ if and only if }b_{n}%
=\sum_{k=0}^{n}\left(  -1\right)  ^{n-k}
\genfrac{\{}{\}}{0pt}{}{n}{k}
_{r}a_{k}\text{ }\left(  n\geq0\right)  .
\end{align*}

\section{Exponential $r$-Lah polynomials}

We first deal with (\ref{3}) for $f\left(  x\right)  =e^{x}$ and present some
identities for arising polynomials. We then rediscover many of the identities
recorded in \cite{NyRa} and present some new relations for $r$-Lah numbers.

It is seen from (\ref{3}) that
\begin{equation}
\left(  xD+2r\right)  ^{\overline{n}}e^{x}=e^{x}\sum_{k=0}^{n}
\genfrac{\lfloor}{\rfloor}{0pt}{}{n}{k}
_{r}x^{k}=e^{x}\mathcal{L}_{n,r}\left(  x\right)  , \label{5}%
\end{equation}
where
\begin{equation}
\mathcal{L}_{n,r}\left(  x\right)  =\sum_{k=0}^{n}
\genfrac{\lfloor}{\rfloor}{0pt}{}{n}{k}
_{r}x^{k}, \label{rLp}%
\end{equation}
which we call \textit{exponential }$r$\textit{-Lah polynomials.} In
particular,
\begin{equation}
\left(  xD\right)  ^{\overline{n}}e^{x}=e^{x}\mathcal{L}_{n,0}\left(
x\right)  =e^{x}\mathcal{L}_{n}\left(  x\right)  . \label{5-a}%
\end{equation}
The notation $\mathcal{L}_{n}\left(  x\right)  $ for this polynomial was
firstly used by Guo and Qi in \cite{Guo} and their related papers.

On the other hand, from (\ref{s1}) and (\ref{2}), we have%
\begin{align}
\left(  xD+r+s\right)  ^{\overline{n}}e^{x}  &  =\sum_{j=0}^{n}
\genfrac{[}{]}{0pt}{}{n}{j}
_{r}\left(  xD+s\right)  ^{j}e^{x}\nonumber\\
&  =e^{x}\sum_{k=0}^{n}\sum_{j=k}^{n}
\genfrac{[}{]}{0pt}{}{n}{j}
_{r}
\genfrac{\{}{\}}{0pt}{}{j}{k}
_{s}x^{k}.\label{5-5}%
\end{align}
Thus, we find from (\ref{5}) and (\ref{5-5}) that
\[%
\genfrac{\lfloor}{\rfloor}{0pt}{}{n}{k}
_{\frac{r+s}{2}}=\sum_{j=k}^{n}
\genfrac{[}{]}{0pt}{}{n}{j}
_{r}
\genfrac{\{}{\}}{0pt}{}{j}{k}
_{s},
\]
which is Theorem 3.11(d) of \cite{NyRa}.

\begin{proposition}
For all non-negative integer $n$,
\begin{equation}
\mathcal{L}_{n+2s,r}\left(  x\right)  =\sum_{k=0}^{n}
\genfrac{\lfloor}{\rfloor}{0pt}{}{n}{k}
_{r+s}x^{k}\mathcal{L}_{2s,r+\frac{k}{2}}\left(  x\right)  . \label{6}%
\end{equation}

\end{proposition}

\begin{proof}
Appealing to (\ref{5}) and noting that $a^{\overline{m+n}}=a^{\overline{m}%
}\left(  a+m\right)  ^{\overline{n}},$ we see that%
\begin{align*}
e^{x}\mathcal{L}_{n+2s,r}\left(  x\right)   &  =\left(  xD+2r\right)
^{\overline{2s}}\left(  xD+2r+2s\right)  ^{\overline{n}}e^{x}\\
&  =\left(  xD+2r\right)  ^{\overline{2s}}e^{x}\mathcal{L}_{n,r+s}\left(
x\right) \\
&  \overset{\text{(\ref{rLp})}}{=}\sum_{k=0}^{n}
\genfrac{\lfloor}{\rfloor}{0pt}{}{n}{k}
_{r+s}\left(  xD+2r\right)  ^{\overline{2s}}\left(  x^{k}e^{x}\right)  .
\end{align*}
Using the Taylor expression of $e^{x}$ in (\ref{5}) and considering that
$\left(  xD+2r\right)  ^{\overline{n}}x^{k}=\left(  k+2r\right)
^{\overline{n}}x^{k},$ we obtain
\begin{equation}
\left(  xD+2r\right)  ^{\overline{n}}\left(  x^{k}e^{x}\right)  =x^{k}%
e^{x}\mathcal{L}_{n,r+k/2}x, \label{5-2}%
\end{equation}
which completes the proof.
\end{proof}

Comparing the coefficients of $x^{k}$ in (\ref{6}), with the use of
(\ref{rLp}), yields to the following relation:

\begin{corollary}%
\[%
\genfrac{\lfloor}{\rfloor}{0pt}{}{n+2s}{m}
_{r}=\sum_{k=0}^{m}
\genfrac{\lfloor}{\rfloor}{0pt}{}{n}{k}
_{r+s}
\genfrac{\lfloor}{\rfloor}{0pt}{}{2s}{m-k}
_{r+\frac{k}{2}}.
\]

\end{corollary}

Multiplying both sides of (\ref{rLp}) with $t^{n}/n!$ and then summing over
$n$ give the following generating function for $\mathcal{L}_{n,r}\left(
x\right)  :$%
\begin{equation}
\frac{1}{\left(  1-t\right)  ^{2r}}e^{x\frac{t}{1-t}}=\sum_{n=0}^{\infty
}\mathcal{L}_{n,r}\left(  x\right)  \frac{t^{n}}{n!}. \label{rLpgf}%
\end{equation}
It is easily seen from (\ref{rLpgf}) that
\begin{align}
\mathcal{L}_{n,r+s}\left(  x+y\right)   &  =\sum_{k=0}^{n}\binom{n}%
{k}\mathcal{L}_{k,r}\left(  x\right)  \mathcal{L}_{n-k,s}\left(  y\right)
,\nonumber\\
\mathcal{L}_{n,r+s}\left(  x\right)   &  =\sum_{k=0}^{n}\binom{n}{k}\left(
2s\right)  ^{\overline{n-k}}\mathcal{L}_{k,r}\left(  x\right)  . \label{14-1}%
\end{align}
As a consequence of (\ref{14-1}), we have (cf. \cite[Theorem 3.4]{NyRa})
\begin{align*}%
\genfrac{\lfloor}{\rfloor}{0pt}{}{n}{l}
_{r+s}  &  =\sum_{k=l}^{n}\binom{n}{k}
\genfrac{\lfloor}{\rfloor}{0pt}{}{k}{l}
\left(  2r+2s\right)  ^{\overline{n-k}}\\
&  =\sum_{k=l}^{n}\binom{n}{k}
\genfrac{\lfloor}{\rfloor}{0pt}{}{k}{l}
_{r}\left(  2s\right)  ^{\overline{n-k}}.
\end{align*}

Let us continue by differentiating both sides of (\ref{rLpgf}) with respect to
$x.$ Taking $m$ times derivative, we see that
\begin{equation}
\frac{d^{m}}{dx^{m}}\mathcal{L}_{n,r}\left(  x\right)  =n^{\underline{m}%
}\mathcal{L}_{n-m,r+\frac{m}{2}}\left(  x\right)  . \label{7-1}%
\end{equation}
On the other hand, we have
\begin{align*}
\sum_{n=0}^{\infty}\frac{d}{dx}\mathcal{L}_{n,r}\left(  x\right)  \frac{t^{n}%
}{n!}  &  =t\sum_{n=0}^{\infty}t^{n}\sum_{n=0}^{\infty}\mathcal{L}%
_{n,r}\left(  x\right)  \frac{t^{n}}{n!}\\
&  =\sum_{n=1}^{\infty}\sum_{k=0}^{n-1}\frac{\mathcal{L}_{k,r}\left(
x\right)  }{k!}t^{n}%
\end{align*}
and then, by (\ref{7-1}),%
\[
\frac{1}{n!}\mathcal{L}_{n,r+\frac{1}{2}}\left(  x\right)  =\sum_{k=0}%
^{n}\frac{\mathcal{L}_{k,r}\left(  x\right)  }{k!}.
\]
Therefore, (\ref{7-1}) implies that%
\[
\frac{1}{n!}n^{\underline{m}}\mathcal{L}_{n-m,r+\frac{m+1}{2}}\left(
x\right)  =\sum_{k=0}^{n-m}\frac{\left(  k+m\right)  ^{\underline{m}}}{\left(
k+m\right)  !}\mathcal{L}_{k,r+\frac{m}{2}}\left(  x\right)  ,
\]
which can be stated as%
\begin{equation}
\frac{1}{n!}\mathcal{L}_{n,r+\frac{m+1}{2}}\left(  x\right)  =\sum_{k=0}%
^{n}\frac{1}{k!}\mathcal{L}_{k,r+\frac{m}{2}}\left(  x\right)  \label{10-}%
\end{equation}
by substituting $n\rightarrow n+m.$ We now compare the coefficients of $x$ to
deduce
\begin{equation}
\binom{n+m}{n}
\genfrac{\lfloor}{\rfloor}{0pt}{}{n}{k}
_{r+\frac{m}{2}}=
\genfrac{\lfloor}{\rfloor}{0pt}{}{n+m}{k+m}
_{r}\binom{k+m}{m}\label{7}%
\end{equation}
and
\[
\frac{1}{n!}
\genfrac{\lfloor}{\rfloor}{0pt}{}{n}{l}
_{r+\frac{m+1}{2}}=\sum_{k=l}^{n}\frac{1}{k!}
\genfrac{\lfloor}{\rfloor}{0pt}{}{k}{l}
_{r+\frac{m}{2}},
\]
from (\ref{7-1}) and (\ref{10-}), respectively. Thus, we arrive at the following:

\begin{proposition}
For all non-negative integers $n$ and $m,$%
\begin{equation}
\frac{1}{\left(  n+m+1\right)  !}
\genfrac{\lfloor}{\rfloor}{0pt}{}{n+m+1}{l+m+1}
_{r}=\sum_{k=l}^{n}\frac{1}{\left(  k+m\right)  !}
\genfrac{\lfloor}{\rfloor}{0pt}{}{k+m}{l+m}
_{r}. \label{10}%
\end{equation}

\end{proposition}

\begin{corollary}
We have%
\begin{equation}
\sum_{k=l}^{p}\frac{k}{\left(  k+m\right)  !}
\genfrac{\lfloor}{\rfloor}{0pt}{}{k+m}{l+m}
_{r}=\frac{p+1}{\left(  p+m+1\right)  !}
\genfrac{\lfloor}{\rfloor}{0pt}{}{p+m+1}{l+m+1}
_{r}-\frac{1}{\left(  p+m+2\right)  !}
\genfrac{\lfloor}{\rfloor}{0pt}{}{p+m+2}{l+m+2}
_{r} \label{10-1}%
\end{equation}
and
\begin{align*}
\sum_{k=l}^{p}\frac{k^{2}}{\left(  k+m\right)  !}
\genfrac{\lfloor}{\rfloor}{0pt}{}{k+m}{l+m}
_{r}  &  =\frac{\left(  p+1\right)  ^{2}}{\left(  p+m+1\right)  !}
\genfrac{\lfloor}{\rfloor}{0pt}{}{p+m+1}{l+m+1}
_{r}-\frac{2p+3}{\left(  p+m+2\right)  !}
\genfrac{\lfloor}{\rfloor}{0pt}{}{p+m+2}{l+m+2}
_{r}\\
&  \quad+\frac{1}{\left(  p+m+3\right)  !}
\genfrac{\lfloor}{\rfloor}{0pt}{}{p+m+3}{l+m+3}
_{r}.
\end{align*}

\end{corollary}

\begin{proof}
Summing both sides of (\ref{10}) over $n,$ we find that%
\begin{align*}
\frac{1}{\left(  p+m+2\right)  !}
\genfrac{\lfloor}{\rfloor}{0pt}{}{p+m+2}{l+m+2}
_{r}  &  =\sum_{k=l}^{p}\sum_{n=k}^{p}\frac{1}{\left(  k+m\right)  !}
\genfrac{\lfloor}{\rfloor}{0pt}{}{k+m}{l+m}
_{r}\\
&  =\frac{p+1}{\left(  p+m+2\right)  !}
\genfrac{\lfloor}{\rfloor}{0pt}{}{p+m+1}{l+m+1}
_{r}-\sum_{k=l}^{p}\frac{k}{\left(  k+m\right)  !}
\genfrac{\lfloor}{\rfloor}{0pt}{}{k+m}{l+m}
_{r},
\end{align*}
which is (\ref{10-1}). The second relation follows from (\ref{10-1}) by
summing over $p.$
\end{proof}

The following proposition offers formula that is analogue to a familiar
Leibniz rule for higher derivatives of the product of two functions.

\begin{proposition}
For $n$-times differentiable functions $f$ and $g,$
\begin{equation}
\left(  xD+2r\right)  ^{\overline{n}}\left[  f\left(  x\right)  g\left(
x\right)  \right]  =\sum_{k=0}^{n}\binom{n}{k}\left[  \left(  xD\right)
^{\underline{k}}f\left(  x\right)  \right]  \left[  \left(  xD+2r+k\right)
^{\overline{n-k}}g\left(  x\right)  \right]  . \label{8}%
\end{equation}

\end{proposition}

\begin{proof}
It follows from (\ref{3}) and (\ref{7}) that%
\begin{align*}
\left(  xD+2r\right)  ^{\overline{n}}\left[  f\left(  x\right)  g\left(
x\right)  \right]   &  =\sum_{l=0}^{n}
\genfrac{\lfloor}{\rfloor}{0pt}{}{n}{l}
_{r}x^{l}\left[  f\left(  x\right)  g\left(  x\right)  \right]  ^{\left(
l\right)  }\\
&  =\sum_{k=0}^{n}x^{k}f^{\left(  k\right)  }\left(  x\right)  \sum
_{l=0}^{n-k}
\genfrac{\lfloor}{\rfloor}{0pt}{}{n-k+k}{l+k}
_{r}\binom{l+k}{k}x^{l}g^{\left(  l\right)  }\left(  x\right) \\
&  =\sum_{k=0}^{n}\binom{n}{k}x^{k}f^{\left(  k\right)  }\left(  x\right)
\sum_{l=0}^{n-k}
\genfrac{\lfloor}{\rfloor}{0pt}{}{n-k}{l}
_{r+k/2}x^{l}g^{\left(  l\right)  }\left(  x\right)  .
\end{align*}
Hence, (\ref{3}) and (\ref{01}) give (\ref{8}).
\end{proof}

Taking $f\left(  x\right)  =x^{m}$ and $g\left(  x\right)  =e^{x}$ in
(\ref{8}), it is seen that
\[
\left(  xD+2r\right)  ^{\overline{n}}\left[  x^{m}e^{x}\right]
\overset{(\text{\ref{5}})}{=}\sum_{k=0}^{n}\binom{n}{k}\left(  m\right)
^{\underline{n-k}}x^{m}e^{x}\mathcal{L}_{k,r+\frac{n-k}{2}}\left(  x\right)
\]
and from (\ref{5-2}) (with substitution $m=2s$)
\begin{equation}
\mathcal{L}_{n,r+s}\left(  x\right)  =\sum_{k=0}^{n}\binom{n}{k}\left(
2s\right)  ^{\underline{n-k}}\mathcal{L}_{k,r+\frac{n-k}{2}}\left(  x\right)
. \label{14-2}%
\end{equation}
Comparing the coefficients of $x^{k}$\ in (\ref{14-2}), with the use of
(\ref{rLp}), we arrive at
\begin{equation}%
\genfrac{\lfloor}{\rfloor}{0pt}{}{n}{l}
_{r+s}=\sum_{k=l}^{n}\binom{n}{k}
\genfrac{\lfloor}{\rfloor}{0pt}{}{k}{l}
_{r+\frac{n-k}{2}}\left(  2s\right)  ^{\underline{n-k}}. \label{9}%
\end{equation}
It emerges that the relation (\ref{9}) is equivalent to \cite[Theorem
3.5]{NyRa} by (\ref{7}):%
\[%
\genfrac{\lfloor}{\rfloor}{0pt}{}{n}{l}
_{r+s}=\sum_{k=l}^{n}
\genfrac{\lfloor}{\rfloor}{0pt}{}{n}{k}
_{r}\binom{k}{l}\left(  2s\right)  ^{\underline{k-l}}.
\]

If we take $f\left(  x\right)  =e^{x}$ and $g\left(  x\right)  =x^{m}$ in
(\ref{8}), then we have
\begin{align*}
x^{m}e^{x}\mathcal{L}_{n,r+\frac{m}{2}}\left(  x\right)   &
\overset{(\text{\ref{5-2}})}{=}\left(  xD+2r\right)  ^{\overline{n}}\left[
e^{x}x^{m}\right] \\
&  \overset{(\text{\ref{8}})}{=}\sum_{k=0}^{n}\binom{n}{k}\left[  \left(
xD\right)  ^{\underline{k}}e^{x}\right]  \left[  \left(  xD+2r+k\right)
^{\overline{n-k}}x^{m}\right] \\
&  \ =\sum_{k=0}^{n}\binom{n}{k}x^{k}e^{x}\left(  m+2r+k\right)
^{\overline{n-k}}x^{m},
\end{align*}
or (by substitution $m=2s$)
\[
\mathcal{L}_{n,r+s}\left(  x\right)  =\sum_{k=0}^{n}\binom{n}{k}\left(
2r+2s+k\right)  ^{\overline{n-k}}x^{k}.
\]
This and (\ref{rLp}) yield that
\begin{equation}%
\genfrac{\lfloor}{\rfloor}{0pt}{}{n}{l}
_{r}=\binom{n}{k}\left(  2r+k\right)  ^{\overline{n-k}}=\frac{n!}{k!}%
\binom{n+2r-1}{k+2r-1}. \label{rLb}%
\end{equation}
This is nothing but Theorem 3.7 of \cite{NyRa}.

The exponential $r$-Lah polynomials satisfy the following three-term
recurrence relation:

\begin{theorem}
We have%
\begin{equation}
\mathcal{L}_{n+1,r}\left(  x\right)  =\left(  2n+2r+x\right)  \mathcal{L}%
_{n,r}\left(  x\right)  -n\left(  2r+n-1\right)  \mathcal{L}_{n-1,r}\left(
x\right)  . \label{14-3}%
\end{equation}

\end{theorem}

\begin{proof}
Differentiating both sides of (\ref{rLpgf}) with respect to $t$ gives%
\begin{align*}
\sum_{n=0}^{\infty}\mathcal{L}_{n+1,r}\left(  x\right)  \frac{t^{n}}{n!}  &
=\frac{2r}{\left(  1-t\right)  ^{2r+1}}e^{x\frac{t}{1-t}}+\frac{x}{\left(
1-t\right)  ^{2r+2}}e^{x\frac{t}{1-t}}\\
&  =\frac{2r\left(  1-t\right)  +x}{\left(  1-t\right)  ^{2}}\sum
_{n=0}^{\infty}\mathcal{L}_{n,r}\left(  x\right)  \frac{t^{n}}{n!}.
\end{align*}
We multiply both sides with $\left(  1-t\right)  ^{2}$\ and arrange to find
\begin{align*}
&  \sum_{n=0}^{\infty}\mathcal{L}_{n+1,r}\left(  x\right)  \frac{t^{n}}%
{n!}-\sum_{n=1}^{\infty}2n\mathcal{L}_{n,r}\left(  x\right)  \frac{t^{n}}%
{n!}+\sum_{n=2}^{\infty}n\left(  n-1\right)  \mathcal{L}_{n-1,r}\left(
x\right)  \frac{t^{n}}{n!}\\
&  \ =\sum_{n=0}^{\infty}\left(  2r+x\right)  \mathcal{L}_{n,r}\left(
x\right)  \frac{t^{n}}{n!}-\sum_{n=1}^{\infty}2rn\mathcal{L}_{n-1,r}\left(
x\right)  \frac{t^{n}}{n!}.
\end{align*}
This gives the desired result.
\end{proof}

Note that (\ref{14-3}) with (\ref{rLp}) and (\ref{rLb}) yields the following
recurrence relation \cite[Theorem 3.1]{NyRa}%
\begin{equation}%
\genfrac{\lfloor}{\rfloor}{0pt}{}{n+1}{k}
_{r}=
\genfrac{\lfloor}{\rfloor}{0pt}{}{n}{k-1}
_{r}+\left(  n+k+2r\right)
\genfrac{\lfloor}{\rfloor}{0pt}{}{n}{k}
_{r},\text{ }1\leq k\leq n. \label{rLrr}%
\end{equation}

\section{Geometric $r$-Lah polynomials}

In this section, we shall present several identities involving hyperharmonic
numbers. These identities follow from the connection between hyperharmonic
numbers and polynomials that appear in (\ref{3}) for $f\left(  x\right)
=\left(  1-x\right)  ^{-1}.$

It is seen from (\ref{3}) that
\[
\left(  xD+2r\right)  ^{\overline{n}}\left(  \frac{1}{1-x}\right)  =\frac
{1}{1-x}\sum_{k=0}^{n}
\genfrac{\lfloor}{\rfloor}{0pt}{}{n}{k}
_{r}k!\left(  \frac{x}{1-x}\right)  ^{k}.
\]
We define
\begin{equation}
\mathfrak{L}_{n,r}\left(  x\right)  =\sum_{k=0}^{n}
\genfrac{\lfloor}{\rfloor}{0pt}{}{n}{k}
_{r}k!x^{k}, \label{rLg}%
\end{equation}
and call these polynomials \textit{geometric }$r$-\textit{Lah polynomials.
}Using the fact $\left(  xD+2r\right)  ^{\overline{n}}x^{m}=\left(
m+2r\right)  ^{\overline{n}}x^{m},$ we see that
\begin{equation}
\left(  xD+2r\right)  ^{\overline{n}}\left(  \frac{1}{1-x}\right)  =\sum
_{m=0}^{\infty}\left(  m+2r\right)  ^{\overline{n}}x^{m}=\frac{1}%
{1-x}\mathfrak{L}_{n,r}\left(  \frac{x}{1-x}\right)  . \label{11}%
\end{equation}
It follows from (\ref{rLb}) and (\ref{rLg}) that $\mathfrak{L}_{n,1/2}\left(
x\right)  =n!\left(  1+x\right)  ^{n},$ in which case (\ref{11}) reduce to the
generating function of rising factorial.

The geometric $r$-Lah polynomials have the following generating function,
which leads to investigate some properties of these polynomials.

\begin{theorem}
We have
\begin{equation}
\sum_{n=0}^{\infty}\mathfrak{L}_{n,r}\left(  x-1\right)  \frac{t^{n}}%
{n!}=\frac{1}{\left(  1-t\right)  ^{2r-1}}\frac{1}{1-xt}. \label{gfrLg}%
\end{equation}

\end{theorem}

\begin{proof}
From (\ref{rLg}) and (\ref{rLgf}), we have%
\begin{align*}
\sum_{n=0}^{\infty}\mathfrak{L}_{n,r}\left(  x\right)  \frac{t^{n}}{n!}  &
\overset{(\text{\ref{rLg}})}{=}\sum_{n=0}^{\infty}\frac{t^{n}}{n!}\sum
_{k=0}^{n}
\genfrac{\lfloor}{\rfloor}{0pt}{}{n}{k}
_{r}k!x^{k}\\
&  \ =\sum_{k=0}^{\infty}x^{k}\left(  k!\sum_{n=k}^{\infty}
\genfrac{\lfloor}{\rfloor}{0pt}{}{n}{k}
_{r}\frac{t^{n}}{n!}\right) \\
&  \;\overset{(\text{\ref{rLgf}})}{=}\left(  \frac{1}{1-t}\right)  ^{2r}%
\sum_{k=0}^{\infty}\left(  \frac{xt}{1-t}\right)  ^{k}\\
&  \ =\left(  \frac{1}{1-t}\right)  ^{2r-1}\frac{1}{1-xt-t},\text{ }\left\vert
\frac{xt}{1-t}\right\vert <1.
\end{align*}
This completes the proof.
\end{proof}

In particular, for $x=0$ in (\ref{gfrLg}), we have
\begin{equation}
\mathfrak{L}_{n,r}\left(  -1\right)  =\sum_{k=0}^{n}\left(  -1\right)  ^{k}
\genfrac{\lfloor}{\rfloor}{0pt}{}{n}{k}
_{r}k!=\left(  2r-1\right)  ^{\overline{n}},\text{ }n\geq0,r\geq1. \label{28}%
\end{equation}
Appealing \cite[Theorem 3.12]{NyRa} gives\textbf{ }%
\[
\sum_{k=0}^{n}\left(  -1\right)  ^{k}
\genfrac{\lfloor}{\rfloor}{0pt}{}{n}{k}
_{r}\left(  2r-1\right)  ^{\overline{k}}=n!.\text{ }%
\]

As an application of (\ref{gfrLg}), we present \textit{the binomial
hyperharmonic identity} with a different formulation from (\ref{bhh}).

\begin{theorem}
\label{bhh-}For all non-negative integers $n$ and $r,$%
\[
h_{n+1}^{\left(  r\right)  }=\sum_{k=0}^{n}\binom{n+r}{k+r}\frac{\left(
-1\right)  ^{k}}{k+1}.
\]

\end{theorem}

\begin{proof}
Let $2r-1\geq0$ be an integer. Integrating both sides of (\ref{gfrLg}) with
respect to $x$ from $0$ to $1,$ we have%
\[
\sum_{n=0}^{\infty}\frac{t^{n}}{n!}\left[  {\displaystyle\int\limits_{0}^{1}}
\mathfrak{L}_{n,r}\left(  x-1\right)  dx\right]  =\frac{1}{t}\frac{-\ln\left(
1-t\right)  }{\left(  1-t\right)  ^{2r-1}}.
\]
Using the generating function of hyperharmonic numbers (\ref{gfhh}), we have
\[
\sum_{n=0}^{\infty}\frac{t^{n}}{n!}\left[  {\displaystyle\int\limits_{0}^{1}}
\mathfrak{L}_{n,r}\left(  x-1\right)  dx\right]  =\sum_{n=0}^{\infty}%
t^{n}h_{n+1}^{\left(  2r-1\right)  }.
\]
Comparing the coefficients of $t^{n}$ gives
\begin{equation}
{\displaystyle\int\limits_{0}^{1}} \mathfrak{L}_{n,r}\left(  x-1\right)
dx=n!h_{n+1}^{\left(  2r-1\right)  }. \label{12}%
\end{equation}
Hence, (\ref{rLg}) and (\ref{rLb}) complete the proof.
\end{proof}

We now write (\ref{gfrLg}) in the form
\[
\sum_{n=0}^{\infty}\mathfrak{L}_{n,r}\left(  x-1\right)  \frac{\left(
1-e^{-t}\right)  ^{n}}{n!}e^{-tm}=e^{-t\left(  m-2r+1\right)  }\frac
{1}{1+x\left(  e^{-t}-1\right)  }%
\]
by setting $t\rightarrow1-e^{-t}$ and then multiplying both sides by $e^{-tm}%
$. We recall the $r$-geometric polynomials defined by the generating function
\cite{Dil2}%
\begin{equation}
\sum_{n=0}^{\infty}w_{n,r}\left(  x\right)  \frac{t^{n}}{n!}=\frac
{1}{1-x\left(  e^{t}-1\right)  }e^{rt}. \label{15}%
\end{equation}
Utilizing (\ref{15}) and the generating function of $r$-Stirling numbers of
the second kind%
\begin{equation}
\sum_{n=k}^{\infty}
\genfrac{\{}{\}}{0pt}{}{n}{k}
_{r}\frac{t^{n}}{n!}=\frac{\left(  e^{t}-1\right)  ^{k}}{k!}e^{rt},
\label{gfr-S}%
\end{equation}
we see that%
\[
\sum_{n=0}^{\infty}\left[  \sum_{k=0}^{n}\left(  -1\right)  ^{k}
\genfrac{\{}{\}}{0pt}{}{n}{k}
_{m}\mathfrak{L}_{k,r}\left(  x-1\right)  \right]  \frac{\left(  -t\right)
^{n}}{n!}=\sum_{n=0}^{\infty}w_{n,m-2r+1}\left(  -x\right)  \frac{\left(
-t\right)  ^{n}}{n!},
\]
which relate the $r$-geometric polynomials and geometric $r$-Lah polynomials
as in the following:

\begin{theorem}
\label{13-1}For all integers $n\geq1$ and $m+1\geq2r\geq1,$ we have%
\begin{equation}
\sum_{k=0}^{n}\left(  -1\right)  ^{k}
\genfrac{\{}{\}}{0pt}{}{n}{k}
_{m}\mathfrak{L}_{k,r}\left(  x-1\right)  =w_{n,m+1-2r}\left(  -x\right)  .
\label{13}%
\end{equation}

\end{theorem}

It should be noted that for $r=1/2$ and $x=1/2,$ (\ref{13}) becomes the well
known identity
\[
\sum_{k=0}^{n}\left(  -1\right)  ^{k}
\genfrac{\{}{\}}{0pt}{}{n}{k}
_{m}\ \frac{k!}{2^{k}}=E_{n}\left(  m\right)  ,
\]
upon the use of $\mathfrak{L}_{k,1/2}\left(  -1/2\right)  =k!/2^{k}$ and
$w_{n,m}\left(  -1/2\right)  =E_{n}\left(  m\right)  .$ Here, $E_{n}\left(
x\right)  $ is the $n$th Euler polynomial \cite[p. 529]{S}.

Next theorem shows that the Bernoulli polynomials can be described as a sum of
the products of hyperharmonic numbers and $r$-Stirling numbers of the second kind.

\begin{theorem}
\label{17-}For all non-negative integers $n,r,m$, we have%
\begin{equation}
\sum_{k=0}^{n}\left(  -1\right)  ^{k}
\genfrac{\{}{\}}{0pt}{}{n}{k}
_{m}k!h_{k+1}^{\left(  r\right)  }=B_{n}\left(  m-r\right)  \label{17}%
\end{equation}
and
\begin{equation}
\sum_{k=0}^{n}
\genfrac{[}{]}{0pt}{}{n}{k}
_{m}B_{k}\left(  r\right)  =n!h_{n+1}^{\left(  r+m-1\right)  }. \label{17a}%
\end{equation}

\end{theorem}

\begin{proof}
Integrating both sides of (\ref{13}) with respect to $x$ from $0$ to $1,$ and
using (\ref{12}) we see that%
\begin{equation}
\sum_{k=0}^{n}\left(  -1\right)  ^{k}
\genfrac{\{}{\}}{0pt}{}{n}{k}%
_{m}k!h_{k+1}^{\left(  2r-1\right)  }=%
{\displaystyle\int\limits_{0}^{1}}
{\displaystyle\int\limits_{0}^{1}} w_{n,m+1-2r}\left(  -x\right)  dx.
\label{16a}%
\end{equation}
We now integrate (\ref{15}) with respect to $x$ from $0$ to $1$ and use the
generating function of Bernoulli polynomials \cite[p. 529]{S}
\[
\sum_{n=0}^{\infty}B_{n}\left(  x\right)  \frac{t^{n}}{n!}=\frac{t}{e^{t}%
-1}e^{xt},
\]
to deduce that
\begin{equation}
{\displaystyle\int\limits_{0}^{1}} w_{n,r}\left(  -x\right)  dx=B_{n}\left(
r\right)  . \label{16}%
\end{equation}
Thus, (\ref{17}) follows from (\ref{16a}) and (\ref{16}).

To obtain (\ref{17a}) we apply $r$-Stirling transform to (\ref{17}) and then
use the well-known formula $B_{k}\left(  1-x\right)  =\left(  -1\right)
^{k}B_{k}\left(  x\right)  .$
\end{proof}

\begin{remark}
\label{Kellerr}$\bullet$ Note that (\ref{16}) is a natural extension of
Keller's identity \cite{Keller}%
\[
{\displaystyle\int\limits_{0}^{1}} w_{n}\left(  -x\right)  dx=B_{n},\text{ }%
\]
where $B_{n}=B_{n}\left(  0\right)  $ is the $n$th Bernoulli number.

$\bullet$ Since $h_{n}^{\left(  0\right)  }=1/n,$ we have the well-known
formula for Bernoulli polynomials%
\[
\sum_{k=0}^{n}
\genfrac{\{}{\}}{0pt}{}{n}{k}
_{m}\frac{\left(  -1\right)  ^{k}}{k+1}k!=B_{n}\left(  m\right)  .
\]

$\bullet$ (\ref{17a}) specializes some formulas in \cite[p. 129]{CaDa}.
\end{remark}

We want to finalize this section giving a connection between the exponential
$r$-Lah polynomials and geometric $r$-Lah polynomials, namely,
\begin{equation}
\mathfrak{L}_{n,r}\left(  x\right)  = {\displaystyle\int\limits_{0}^{\infty}}
e^{-\lambda}\mathcal{L}_{n,r}\left(  x\lambda\right)  d\lambda. \label{14-4}%
\end{equation}
This connection follows from (\ref{rLp}), (\ref{rLg}) and the well-known
identity%
\[
{\displaystyle\int\limits_{0}^{\infty}} z^{k}e^{-z}dz=k!,\text{ }%
k\in\mathbb{N} .
\]
Then, with the use of (\ref{12}), we see that this connection leads some
identities for the hyperharmonic numbers:

\begin{theorem}
We have%
\begin{align*}
\left(  n+1\right)  h_{n+1}^{\left(  r\right)  }  &  =\left(  n+r\right)
h_{n}^{\left(  r\right)  }+\frac{r^{\overline{n}}}{n!},\\
h_{n+1}^{\left(  r+s\right)  }  &  =\sum_{k=0}^{n}\binom{n-k+s}{s}%
h_{k+1}^{\left(  r-1\right)  }%
\end{align*}
and%
\[
h_{n+1}^{\left(  r+s\right)  }=\sum_{k=0}^{\min\left(  n,s\right)  }\binom
{s}{k}h_{n-k+1}^{\left(  r+k\right)  }.
\]

\end{theorem}

\begin{proof}
To prove the first identity, we replace $x$ by $x\lambda$ in (\ref{14-3}) and
multiply both sides by $e^{-\lambda}.$ We then integrate with respect to
$\lambda$ from $0$ to $\infty,$ with the use of (\ref{14-4}), and obtain that
\begin{align}
&  \mathfrak{L}_{n+1,r}\left(  x\right) \nonumber\\
&  =\left(  2n+2r\right)  \mathfrak{L}_{n,r}\left(  x\right)  -n\left(
n+2r-1\right)  \mathfrak{L}_{n-1,r}\left(  x\right)  + {\displaystyle\int%
\limits_{0}^{\infty}} x\lambda\mathcal{L}_{n,r}\left(  x\lambda\right)
e^{-\lambda}d\lambda. \label{14-6}%
\end{align}
It is clear from (\ref{rLp}) that
\[
{\displaystyle\int\limits_{0}^{\infty}} x\lambda\mathcal{L}_{n,r}\left(
x\lambda\right)  e^{-\lambda}d\lambda=\sum_{k=0}^{n}
\genfrac{\lfloor}{\rfloor}{0pt}{}{n}{k}
_{r}x^{k+1}\left(  k+1\right)  !.
\]
We now integrate both sides of (\ref{14-6}) with respect to $x$ from $-1$ to
$0$ and use (\ref{12}) to deduce that%
\begin{align}
\left(  n+1\right)  !h_{n+2}^{\left(  2r-1\right)  }  &  =\left(
2n+2r\right)  n!h_{n+1}^{\left(  2r-1\right)  }-\left(  n+2r-1\right)
n!h_{n}^{\left(  2r-1\right)  }\nonumber\\
&  \quad+\sum_{k=0}^{n}
\genfrac{\lfloor}{\rfloor}{0pt}{}{n}{k}
_{r}\frac{\left(  -1\right)  ^{k+1}}{k+2}\left(  k+1\right)  !. \label{30}%
\end{align}
Now utilizing (\ref{rLrr}), (\ref{12}) and (\ref{28}), we find that%
\begin{align*}
\sum_{k=0}^{n}
\genfrac{\lfloor}{\rfloor}{0pt}{}{n}{k}
_{r}\frac{\left(  -1\right)  ^{k+1}}{k+2}\left(  k+1\right)  !  &  =\sum
_{k=1}^{n+1}
\genfrac{\lfloor}{\rfloor}{0pt}{}{n+1}{k}
_{r}\frac{\left(  -1\right)  ^{k}}{k+1}k!-\sum_{k=1}^{n}
\genfrac{\lfloor}{\rfloor}{0pt}{}{n}{k}
_{r}\left(  n+2r+k\right)  \frac{\left(  -1\right)  ^{k}}{k+1}k!\\
&  =\left(  n+1\right)  !h_{n+2}^{\left(  2r-1\right)  }-
\genfrac{\lfloor}{\rfloor}{0pt}{}{n+1}{0}
_{r}-\left(  n+2r-1\right)  n!h_{n+1}^{\left(  2r-1\right)  }\\
&  \quad+\left(  n+2r\right)
\genfrac{\lfloor}{\rfloor}{0pt}{}{n}{0}
_{r}-\left(  2r-1\right)  ^{\overline{n}}\\
&  =\left(  n+1\right)  !h_{n+2}^{\left(  2r-1\right)  }-\left(
n+2r-1\right)  n!h_{n+1}^{\left(  2r-1\right)  }-\left(  2r-1\right)
^{\overline{n}}.
\end{align*}
Hence, (\ref{30}) completes the proof of the first identity.

Proofs of the second and the third identities are similar, but for this time
we use (\ref{14-1}) and (\ref{14-2}) instead of (\ref{14-3}), respectively.
\end{proof}

It is worth noting that the first and second identities are proved in the
recent paper \cite{DilMun} with a different method.

\section{Harmonic geometric $r$-Lah polynomials}

We continue to present identities for hyperharmonic numbers, such as
generating function with respect to upper index, generalizations of the
symmetric formula (\ref{hsf}), formulas for the sum of the products of
hyperharmonic numbers.

Setting $f\left(  x\right)  =-\ln\left(  1-x\right)  /\left(  1-x\right)  $ in
(\ref{3}) and using \cite[Eq. (27)]{Dil}
\[
\frac{d^{k}}{dx^{k}}\left(  \frac{-\ln\left(  1-x\right)  }{1-x}\right)
=k!\frac{H_{k}-\ln\left(  1-x\right)  }{\left(  1-x\right)  ^{k+1}},
\]
we deduce that%
\begin{align*}
&  \left(  xD+2r\right)  ^{\overline{n}}\left(  \frac{-\ln\left(  1-x\right)
}{1-x}\right) \\
&  =\frac{1}{1-x}\sum_{k=0}^{n}
\genfrac{\lfloor}{\rfloor}{0pt}{}{n}{k}
_{r}H_{k}k!\left(  \frac{x}{1-x}\right)  ^{k}-\frac{\ln\left(  1-x\right)
}{1-x}\mathfrak{L}_{n,r}\left(  \frac{x}{1-x}\right)  .
\end{align*}
Let $_{H}\mathfrak{L}_{n,r}\left(  x\right)  $ denote the sum in the
right-hand side of the above equation, i.e.,
\begin{equation}
_{H}\mathfrak{L}_{n,r}\left(  x\right)  =\sum_{k=0}^{n}
\genfrac{\lfloor}{\rfloor}{0pt}{}{n}{k}
_{r}H_{k}k!x^{k}, \label{19}%
\end{equation}
which we call \textit{harmonic geometric }$r$\textit{-Lah polynomials}. Then,
considering the generating function of Harmonic numbers (\ref{gfhh}), we
arrive at a closed form evaluation formula for power series involving harmonic numbers.

\begin{theorem}
For all non-negative integers $n,r$%
\[
\sum_{m=0}^{\infty}\left(  m+2r\right)  ^{\overline{n}}H_{m}x^{m}=\frac
{1}{1-x}\text{ }_{H}\mathfrak{L}_{n,r}\left(  \frac{x}{1-x}\right)
\mathfrak{-}\frac{\ln\left(  1-x\right)  }{1-x}\mathfrak{L}_{n,r}\left(
\frac{x}{1-x}\right)  .
\]

\end{theorem}

For $r=1/2,$ the above relation can be written as%
\[
\sum_{m=0}^{\infty}\binom{m+n}{n}H_{m}x^{m}=\frac{1}{1-x}\sum_{k=0}^{n}%
\binom{n}{k}\frac{H_{k}x^{k}}{\left(  1-x\right)  ^{k}}\mathfrak{-}\frac
{\ln\left(  1-x\right)  }{\left(  1-x\right)  ^{n+1}}.
\]
We use (\ref{gfhh}) and (\ref{hh-h}) to see that
\begin{align*}
\sum_{m=0}^{\infty}\binom{m+n}{n}H_{m}x^{m}  &  =\frac{1}{1-x}\sum_{k=0}%
^{n}\binom{n}{k}\frac{H_{k}x^{k}}{\left(  1-x\right)  ^{k}}\\
&  +\sum_{m=0}^{\infty}\binom{m+n}{n}H_{m+n}x^{m}-\sum_{m=0}^{\infty}%
\binom{m+n}{n}H_{n}x^{m}.
\end{align*}
Considering this formula as
\[
\sum_{m=0}^{\infty}\binom{m+n}{m}\left(  H_{n+m}-H_{m}\right)  x^{m}=H_{n}%
\sum_{m=0}^{\infty}\binom{m+n}{n}x^{m}-\frac{1}{1-x}\sum_{k=0}^{n}\binom{n}%
{k}\frac{H_{k}x^{k}}{\left(  1-x\right)  ^{k}},
\]
then using (\ref{hh-h}) and
\begin{equation}
\left(  1-t\right)  ^{-\alpha}=\sum\limits_{n=0}^{\infty}\frac{\left(
\alpha\right)  ^{\overline{n}}}{n!}t^{n}, \label{32}%
\end{equation}
we have
\begin{equation}
\sum_{m=0}^{\infty}h_{n}^{\left(  m+1\right)  }x^{m}=\frac{H_{n}}{\left(
1-x\right)  ^{n+1}}-\frac{1}{1-x}\sum_{k=0}^{n}\binom{n}{k}H_{k}\left(
\frac{x}{1-x}\right)  ^{k}. \label{31}%
\end{equation}
Using \cite[Corollary 8]{B1}%
\[
\sum_{k=0}^{n}\binom{n}{k}H_{k}\lambda^{k}=\left(  1+\lambda\right)  ^{n}%
H_{n}-\sum_{j=1}^{n}\frac{1}{j}\left(  1+\lambda\right)  ^{n-j}%
\]
in the last part of (\ref{31}) for $\lambda=x/\left(  1-x\right)  $, we obtain
a generating function for hyperharmonic numbers with respect to upper index:

\begin{theorem}
We have
\[
\sum_{m=0}^{\infty}h_{n}^{\left(  m+1\right)  }x^{m}=\sum_{j=1}^{n-1}\frac
{1}{n-j}\left(  \frac{1}{1-x}\right)  ^{j+1}.
\]

\end{theorem}

It is good to note that using (\ref{32}) gives%
\begin{equation}
\sum_{m=0}^{\infty}h_{n}^{\left(  m+1\right)  }x^{m}=\sum_{j=0}^{n-1}\frac
{1}{n-j}\sum_{m=0}^{\infty}\frac{\left(  j+1\right)  ^{\overline{m}}}{m!}%
x^{m}. \label{33}%
\end{equation}
Comparing the coefficients of $x^{m}$ in the above equation, we deduce%
\[
h_{n}^{\left(  m+1\right)  }=\sum_{j=1}^{n}\binom{n+m-j}{m}\frac{1}{j},
\]
which was proved in \cite{Benjamin,DilandMezo} by different methods.

Moreover, integrate both sides of (\ref{33}) with respect to $x$ from $0$ to
$x$ and multiply it by $1/x.$ Repeat this procedure for $q$ times to obtain
\[
\sum_{m=0}^{\infty}\frac{h_{n}^{\left(  m+1\right)  }}{\left(  m+1\right)
^{q}}x^{m}=\sum_{j=0}^{n-1}\frac{1}{n-j}\sum_{m=0}^{\infty}\frac{\left(
j+1\right)  ^{\overline{m}}}{m!}\frac{x^{m}}{\left(  m+1\right)  ^{q}}.
\]
Then we have obtained the following closed-form evaluation formula for a
Euler-type sum:

\begin{theorem}
\label{upi}%
\[
\sum_{m=0}^{\infty}\frac{h_{n}^{\left(  m+1\right)  }}{\left(  m+1\right)
^{q}}x^{m}=\sum_{j=0}^{n-1}\frac{1}{n-j}\Phi_{j+1}^{\ast}\left(  x,q,1\right)
,
\]
where
\[
\Phi_{\mu}^{\ast}\left(  z,s,a\right)  =\sum_{m=0}^{\infty}\frac{\left(
\mu\right)  ^{\overline{m}}}{m!}\frac{z^{m}}{\left(  m+a\right)  ^{s}}%
\]
is a generalization of the Hurwitz--Lerch Zeta function \cite{Goyal}.
\end{theorem}

One of the generalizations of the symmetric formula (\ref{hsf}) is as follows:

\begin{theorem}
\label{27-}For all integers $n\geq0$ and $r\geq1,$%
\begin{equation}
h_{n}^{\left(  r\right)  }=\sum_{k=0}^{n}\left(  -1\right)  ^{k+1}\binom
{n+r}{k+r}H_{k}. \label{27}%
\end{equation}

\end{theorem}

\begin{proof}
From (\ref{19}), (\ref{rLgf}) and (\ref{gfhh}) it is seen that%
\begin{align*}
\sum_{n=0}^{\infty}\text{ }_{H}\mathfrak{L}_{n,r}\left(  x\right)  \frac
{t^{n}}{n!}  &  \overset{(\text{\ref{19}})}{=}\sum_{n=0}^{\infty}\frac{t^{n}%
}{n!}\sum_{k=0}^{n}
\genfrac{\lfloor}{\rfloor}{0pt}{}{n}{k}
_{r}k!H_{k}x^{k}\\
&  \ =\sum_{k=0}^{\infty}k!H_{k}x^{k}\sum_{n=k}^{\infty}
\genfrac{\lfloor}{\rfloor}{0pt}{}{n}{k}
_{r}\frac{t^{n}}{n!}\\
&  \;\overset{(\text{\ref{rLgf}})}{=}\frac{1}{\left(  1-t\right)  ^{2r}}%
\sum_{k=0}^{\infty}H_{k}\left(  \frac{xt}{1-t}\right)  ^{k}\\
&  \;\overset{(\text{\ref{gfhh}})}{=}\frac{1}{\left(  1-t\right)  ^{2r-1}%
}\frac{\ln\left(  1-t\right)  -\ln\left(  1-xt-t\right)  }{1-xt-t},\left\vert
\frac{xt}{1-t}\right\vert <1.
\end{align*}
Then we have the generating function for the harmonic geometric $r$-Lah
polynomials%
\begin{equation}
\sum_{n=0}^{\infty}\text{ }_{H}\mathfrak{L}_{n,r}\left(  x-1\right)
\frac{t^{n}}{n!}=\frac{\ln\left(  1-t\right)  -\ln\left(  1-xt\right)
}{\left(  1-t\right)  ^{2r-1}\left(  1-xt\right)  }. \label{18}%
\end{equation}
Comparing (\ref{gfhh}) and (\ref{18}) with $x=0$, we reach that
\[
_{H}\mathfrak{L}_{n,r}\left(  -1\right)  =-n!h_{n}^{\left(  2r-1\right)  }.
\]
Utilizing (\ref{19}) and (\ref{rLb}) give (\ref{27}).
\end{proof}

In the following theorem, we give a generalization both of (\ref{27}) and the
symmetric formula (\ref{hsf}).

\begin{theorem}
\label{29-}%
\begin{equation}
\sum_{k=0}^{n}\left(  -1\right)  ^{k+m+1}\binom{n+m+r}{k+m+r}\binom{k+m}%
{m}H_{k+m}=h_{n}^{\left(  r\right)  }-\binom{n+r-1}{n}H_{m}. \label{29}%
\end{equation}
In particular,%
\[
\sum_{k=0}^{n}\binom{n}{k}\frac{\left(  -1\right)  ^{k+m+1}}{k+m+1}%
H_{k+m}=\frac{1}{n+m+1}\binom{n+m}{m}^{-1}\left(  H_{n}-H_{m}\right)
\]
and%
\[
\sum_{k=0}^{n}\left(  -1\right)  ^{k+r+1}
\genfrac{\lfloor}{\rfloor}{0pt}{}{n}{k}
_{r+1}\left(  k+r\right)  !H_{k+r}=\left(  n+r\right)  !\left(  H_{n+r}%
-2H_{r}\right)  .
\]

\end{theorem}

\begin{proof}
By induction on $m,$ it can be shown that%
\[
\sum_{n=m}^{\infty}\frac{d^{m}}{dx^{m}}\text{ }_{H}\mathfrak{L}_{n,r}\left(
x-1\right)  \frac{t^{n}}{n!}=\left(  -1\right)  ^{m}m!\frac{\ln\left(
1-t\right)  -\ln\left(  1-xt\right)  +H_{m}}{\left(  1-t\right)
^{2r-1}\left(  1-xt\right)  ^{m+1}}t^{m}.
\]
Thus, we have
\[
\sum_{n=0}^{\infty}\frac{t^{n}}{\left(  n+m\right)  !}\left.  \frac{d^{m}%
}{dx^{m}}\text{ }_{H}\mathfrak{L}_{n+m,r}\left(  x-1\right)  \right\vert
_{x=0}=\left(  -1\right)  ^{m}m!\frac{\ln\left(  1-t\right)  +H_{m}}{\left(
1-t\right)  ^{2r-1}}.
\]
Using (\ref{gfhh}) and (\ref{32}) in the above equation and then comparing the
coefficients, we obtain
\[
\left.  \frac{d^{m}}{dx^{m}}\text{ }\frac{_{H}\mathfrak{L}_{n+m,r}\left(
x-1\right)  }{\left(  n+m\right)  !}\right\vert _{x=0}=\left(  -1\right)
^{m+1}m!\left(  h_{n}^{\left(  2r-1\right)  }-H_{m}\frac{\left(  2r-1\right)
^{\overline{n}}}{n!}\right)  .
\]
From (\ref{19}) and (\ref{rLb}), we have (\ref{29}).
\end{proof}

To investigate the relation between the harmonic geometric $r$-Lah polynomials
and some other well-known numbers or polynomials, we recall the harmonic
$r$-geometric polynomials, defined by \cite{Dil2}%
\begin{equation}
_{H}w_{n,r}\left(  x\right)  =\sum_{k=0}^{n}
\genfrac{\{}{\}}{0pt}{}{n}{k}
_{r}k!H_{k}x^{k}. \label{20}%
\end{equation}
The harmonic $r$-geometric polynomials have the following generating function
and a relation with Bernoulli polynomials.

\begin{lemma}
The harmonic $r$-geometric polynomials have the following identities%
\begin{align}
\sum_{n=0}^{\infty}\text{ }_{H}w_{n,r}\left(  x\right)  \frac{t^{n}}{n!}  &
=\frac{-\ln\left(  1-x\left(  e^{t}-1\right)  \right)  }{1-x\left(
e^{t}-1\right)  }e^{rt},\label{25}\\
{\displaystyle\int\limits_{0}^{1}} \text{ }_{H}w_{n,r}\left(  -x\right)  dx
&  =\frac{-n}{2}B_{n-1}\left(  r\right)  ,\text{ }n\geq1,\text{ }r\geq0.
\label{26}%
\end{align}

\end{lemma}

\begin{proof}
From (\ref{20}), we have%
\begin{align*}
\sum_{n=0}^{\infty}\text{ }_{H}w_{n,r}\left(  x\right)  \frac{t^{n}}{n!}  &
\ =\sum_{n=0}^{\infty}\frac{t^{n}}{n!}\left[  \sum_{k=0}^{n}
\genfrac{\{}{\}}{0pt}{}{n}{k}
_{r}k!H_{k}x^{k}\right] \\
&  \ =\sum_{k=0}^{\infty}k!H_{k}x^{k}\sum_{n=k}^{\infty}
\genfrac{\{}{\}}{0pt}{}{n}{k}
_{r}\frac{t^{n}}{n!}\\
&  \overset{\left(  \text{\ref{gfr-S}}\right)  }{=}e^{rt}\sum_{k=0}^{\infty
}H_{k}\left(  x\left(  e^{t}-1\right)  \right)  ^{k}\\
&  \overset{\left(  \text{\ref{gfhh}}\right)  }{=}\frac{-\ln\left(  1-x\left(
e^{t}-1\right)  \right)  }{1-x\left(  e^{t}-1\right)  }e^{rt}.
\end{align*}

The proof of (\ref{26}) is similar to that of (\ref{16}) and is omitted.
\end{proof}

Remark that the relation (\ref{26}) is a generalization of the second identity
given in \cite[Theorem 1.3 ]{Keller}.

The following theorem present a relationship between the harmonic geometric
$r$-Lah, $r$-geometric and harmonic $r$-geometric polynomials. The proof is
similar to that of Theorem \ref{13-1}, so we omit it.

\begin{theorem}
For all integers $n\geq1$ and $m+1\geq2r\geq1$%
\begin{equation}
\sum_{k=0}^{n}\left(  -1\right)  ^{k}
\genfrac{\{}{\}}{0pt}{}{n}{k}
_{m}\text{ }_{H}\mathfrak{L}_{k,r}\left(  x-1\right)  =nw_{n-1,m-2r+1}\left(
-x\right)  +\text{ }_{H}w_{n,m-2r+1}\left(  -x\right)  . \label{21}%
\end{equation}

\end{theorem}

Now, we want to deal with some applications of this theorem. Since
$w_{n,r}\left(  0\right)  =r^{n},$ $_{H}w_{n,r}\left(  0\right)  =0$ and
$_{H}\mathfrak{L}_{n,r}\left(  -1\right)  =-n!h_{n}^{\left(  2r-1\right)  },$
taking $x=0$ in (\ref{21}) gives (\ref{22}) in the following corollary.
Moreover, applying $r$-Stirling transform to (\ref{22}) yields (\ref{23})
which is also given in \cite{CaDa} by different means.

\begin{corollary}
We have%
\begin{align}
\sum_{k=1}^{n}\left(  -1\right)  ^{k+1}
\genfrac{\{}{\}}{0pt}{}{n}{k}
_{m}k!h_{k}^{\left(  r\right)  }  &  =n\left(  m-r\right)  ^{n-1},\label{22}\\
\sum_{k=1}^{n}k
\genfrac{[}{]}{0pt}{}{n}{k}
_{m}r^{k-1}  &  =n!h_{n}^{\left(  r+m\right)  }. \label{23}%
\end{align}

\end{corollary}

\begin{theorem}
\label{24-}For all integers $n,r,s\geq0$ \textbf{ }%
\begin{equation}
\sum_{k=1}^{n}h_{k}^{\left(  r\right)  }h_{n+1-k}^{\left(  s\right)  }%
=2\sum_{k=1}^{n}\binom{n+r+s}{k+r+s}\frac{\left(  -1\right)  ^{k+1}}{k+1}H_{k}
\label{24}%
\end{equation}
and
\begin{equation}
\sum_{l=0}^{n}h_{l+1}^{\left(  r-1\right)  }h_{n-l}^{\left(  r\right)
}\mathbf{=}\frac{1}{n!}\sum_{k=1}^{n}\left(  -1\right)  ^{k+1}
\genfrac{[}{]}{0pt}{}{n}{k}
_{m}kB_{k-1}\left(  m-2r+1\right)  \label{24-1}%
\end{equation}

\end{theorem}

\begin{proof}
Integrating both sides of (\ref{18}) with respect to $x$ from $0$ to $1,$ we
have%
\[
\sum_{n=0}^{\infty}\frac{t^{n}}{n!} {\displaystyle\int\limits_{0}^{1}} \text{
}_{H}\mathfrak{L}_{n,r+s}\left(  x-1\right)  dx=\frac{-1}{2t}\frac{\ln
^{2}\left(  1-t\right)  }{\left(  1-t\right)  ^{2r+2s-1}}.
\]
Using (\ref{gfhh}) in the above equation yields%
\begin{align*}
\sum_{n=0}^{\infty}\frac{t^{n+1}}{n!} {\displaystyle\int\limits_{0}^{1}}
\text{ }_{H}\mathfrak{L}_{n,r+s}\left(  x-1\right)  dx  &  =\frac{-1}{2}%
\sum_{n=0}^{\infty}t^{n}h_{n}^{\left(  2s\right)  }\sum_{k=0}^{\infty}%
t^{k}h_{k}^{\left(  2r-1\right)  }\\
&  =\frac{-1}{2}\sum_{n=0}^{\infty}t^{n}\left[  \sum_{k=0}^{n}h_{k}^{\left(
2r-1\right)  }h_{n-k}^{\left(  2s\right)  }\right]  .
\end{align*}
Comparing the coefficients of $t^{n}$ gives%
\begin{equation}
{\displaystyle\int\limits_{0}^{1}} \text{ }_{H}\mathfrak{L}_{n-1,r+s}\left(
x-1\right)  dx=-\frac{\left(  n-1\right)  !}{2}\sum_{k=1}^{n-1}h_{k}^{\left(
2r-1\right)  }h_{n-k}^{\left(  2s\right)  }. \label{34}%
\end{equation}
On the other hand, it is seen from (\ref{19}) that
\[
{\displaystyle\int\limits_{0}^{1}} \text{ }_{H}\mathfrak{L}_{n-1,r+s}\left(
x-1\right)  dx=-\sum_{k=0}^{n-1}
\genfrac{\lfloor}{\rfloor}{0pt}{}{n-1}{k}%
_{r+s}\frac{\left(  -1\right)  }{k+1}^{k+1}H_{k}k!.
\]
Hence, (\ref{24}) follows from (\ref{rLb}).

To prove (\ref{24-1}) we first integrate both sides of (\ref{21}) with respect
to $x$ from $0$ to $1$ and use (\ref{16}), (\ref{26}) and (\ref{34}). Then we
have%
\[
\sum_{k=1}^{n}\left(  -1\right)  ^{k+1}
\genfrac{\{}{\}}{0pt}{}{n}{k}
_{m}k!\sum_{l=1}^{k}h_{l}^{\left(  r-1\right)  }h_{k+1-l}^{\left(  r\right)
}=nB_{n-1}\left(  m-2r+1\right)  .
\]
Hence, $r$-Stirling transform\textbf{ }implies (\ref{24-1}).
\end{proof}

Particular cases $r=1$ in (\ref{24-1}) and $r=1,$ $s=0$ in (\ref{24}) give
\begin{align}
\sum_{k=1}^{n}\frac{H_{k}}{n+1-k}  &  =\frac{1}{n!}\sum_{k=1}^{n}\left(
-1\right)  ^{k+1}
\genfrac{[}{]}{0pt}{}{n}{k}
_{m}kB_{k-1}\left(  m-1\right) \nonumber\\
&  =2\sum_{k=1}^{n}\binom{n+1}{k+1}\frac{\left(  -1\right)  ^{k+1}}{k+1}H_{k},
\label{40}%
\end{align}
respectively. However, the last formula can be evaluated as follows:

\begin{proposition}
For all integers $n\geq1,$%
\[
\sum_{k=1}^{n}\binom{n+1}{k+1}\frac{\left(  -1\right)  ^{k+1}}{k+1}H_{k}%
=\frac{1}{2}\left\{  \left(  H_{n+1}\right)  ^{2}-H_{n+1}^{\left(  2\right)
}\right\}  .
\]

\end{proposition}

\begin{proof}
We have
\begin{align}
&  \sum_{k=1}^{n}\binom{n+1}{k+1}\frac{\left(  -1\right)  ^{k+1}}{k+1}%
H_{k}\nonumber\\
&  =\sum_{k=1}^{n-1}\left\{  \binom{n}{k+1}+\binom{n}{k}\right\}
\frac{\left(  -1\right)  ^{k+1}}{k+1}H_{k}+\frac{\left(  -1\right)  ^{n+1}%
}{n+1}H_{n}\nonumber\\
&  =n\sum_{k=1}^{n-1}\binom{n-1}{k}\frac{\left(  -1\right)  ^{k+1}}{\left(
k+1\right)  ^{2}}H_{k}+\sum_{k=1}^{n}\binom{n}{k}\frac{\left(  -1\right)
^{k+1}}{k+1}H_{k}. \label{39}%
\end{align}
If we set
\[
A\left(  n\right)  =\sum_{k=1}^{n}\binom{n+1}{k+1}\frac{\left(  -1\right)
^{k+1}}{k+1}H_{k}=\left(  n+1\right)  \sum_{k=1}^{n}\binom{n}{k}\frac{\left(
-1\right)  ^{k+1}}{\left(  k+1\right)  ^{2}}H_{k}%
\]
and use (\ref{hsf}), then (\ref{39}) becomes%
\begin{align}
A\left(  n\right)   &  =\frac{H_{n}}{n+1}+A\left(  n-1\right)  =\sum_{k=1}%
^{n}\frac{H_{k}}{k+1}\label{41}\\
&  =\sum_{k=1}^{n}\frac{1}{k+1}\left(  H_{k+1}-\frac{1}{k+1}\right)
\nonumber\\
&  =\frac{1}{2}\left\{  \left(  H_{n+1}\right)  ^{2}+H_{n+1}^{\left(
2\right)  }\right\}  -H_{n+1}^{\left(  2\right)  },\nonumber
\end{align}
where we have used that
\begin{align*}
\sum_{k=1}^{n}\frac{1}{k}H_{k}  &  =\sum_{j=1}^{n}\frac{1}{j}\sum_{k=j}%
^{n}\frac{1}{k}=\sum_{j=1}^{n}\frac{1}{j}\left(  H_{n}-H_{j}+\frac{1}%
{j}\right) \\
&  =\left(  H_{n}\right)  ^{2}+H_{n}^{\left(  2\right)  }-\sum_{j=1}^{n}%
\frac{1}{j}H_{j}.
\end{align*}
This completes the proof.
\end{proof}

Combining (\ref{40}) and (\ref{41}), we conclude that%
\[
2\sum_{k=1}^{n}\frac{H_{k}}{k+1}=\sum_{k=1}^{n}\frac{H_{k}}{n+1-k}=\left(
H_{n+1}\right)  ^{2}-H_{n+1}^{\left(  2\right)  }%
\]
and%
\[
\frac{1}{n!}\sum_{k=1}^{n}\left(  -1\right)  ^{k+1}
\genfrac{[}{]}{0pt}{}{n}{k}
_{m}kB_{k-1}\left(  m-1\right)  =\left(  H_{n+1}\right)  ^{2}-H_{n+1}^{\left(
2\right)  },
\]
and then, from the $r$-Stirling transform
\[
\sum_{k=0}^{n}\left(  -1\right)  ^{k+1}
\genfrac{\{}{\}}{0pt}{}{n}{k}
_{m}k!\left[  \left(  H_{k+1}\right)  ^{2}-H_{k+1}^{\left(  2\right)
}\right]  =nB_{n-1}\left(  m-1\right)  .
\]


\begin{thebibliography}{99}                                                                                               %


\bibitem {Ahuja}Ahuja J.C., Enneking E.A.: Concavity property and a recurrence
relation for associated Lah numbers. Fibonacci Quart. \textbf{17}(2), 158--161 (1979)

\bibitem {Barry}Barry P.: Some observations on the Lah and Laguerre transforms
of integer sequences. J. Integer Seq. \textbf{10}, Article 07.4.6 (2007)

\bibitem {Belbachir}Belbachir H., Belkhir A.: Cross recurrence relations for
$r$-Lah numbers. Ars Combin. \textbf{110} 199--203 (2013)

\bibitem {Benjamin}Benjamin A.T., Gaebler D., Gaebler R.: A combinatorial
approach to hyperharmonic numbers. Integers \textbf{3}, \#A15 (2003)

\bibitem {B3}Boyadzhiev K.N.: A series transformation formula and related
polynomials. Int. J. Math. Math. Sci. \textbf{23}, 3849--3866 (2005)

\bibitem {B4}Boyadzhiev K.N.: Apostol-Bernoulli functions, derivative
polynomials and Eulerian polynomials. Adv. Appl. Discrete Math. \textbf{1},
109--122 (2008)

\bibitem {B2}Boyadzhiev K.N.: Exponential polynomials, Stirling numbers and
evaluation of some gamma integrals. Abstr. Appl. Anal. \textbf{2009}, Article
ID 168672 (2009)

\bibitem {B1}Boyadzhiev K.N.: Harmonic number identities via Euler's
transform. J. Integer Seq. \textbf{12}, Article 09.6.1 (2009)

\bibitem {Boyadzhiev}Boyadzhiev K.N.: Binomial transform and the backward
difference. Adv. Appl. Discrete Math. \textbf{13}(1), 43--63 (2014)

\bibitem {B5}Boyadzhiev K.N.: Lah numbers, Laguerre polynomials of order
negative one, and the $n$th derivative of $\exp\left(  1/x\right)  .$ Acta
Univ. Sapientiae, Mathematica, \textbf{8}(1), 22--31 (2016)

\bibitem {B6}Boyadzhiev K.N.: Notes on the Binomial Transform. World
Scientific, Singapore (2018)

\bibitem {BandDil}Boyadzhiev K.N., Dil A.: Geometric polynomials: properties
and applications to series with zeta values. Analysis Math. \textbf{42}(3),
203--224 (2016)

\bibitem {Broder}Broder A.Z.: The $r$-Stirling numbers. Disc. Math.
\textbf{49}, 241--259 (1984)

\bibitem {CaDa}Can M., Da\u{g}l\i\ M.C.: Extended Bernoulli and Stirling
matrices and related combinatorial identities. Linear Algebra Appl.
\textbf{444}, 114--131 (2014)

\bibitem {Chen}Cheon G.-S., Jung J.-H.: $r$-Whitney numbers of Dowling
lattices. Discrete Math. \textbf{312}, 2337--2348 (2012)

\bibitem {CG}Conway J.H., Guy R.K.: The book of numbers. Springer-Verlag, New
York (1996)

\bibitem {DilandMezo}Dil A., Mez\"{o} \.{I}.: A symmetric algorithm for
hyperharmonic and Fibonacci numbers. Appl. Math. Comp. \textbf{206}, 942--951 (2008)

\bibitem {DM}Dil A., Mez\"{o} I.: Hyperharmonic series involving Hurwitz zeta
function. J. Number Theory \textbf{130}, 360--369 (2010)

\bibitem {Dil}Dil A., Kurt V.: Polynomials related to harmonic numbers and
evaluation of harmonic number series I. Integers \textbf{12}, 1--18 (2012)

\bibitem {Dil2}Dil A., Kurt V.: Polynomials related to harmonic numbers and
evaluation of harmonic number series II.Appl. Anal. Discrete Math. \textbf{5},
212--229 (2011)

\bibitem {DB}Dil A., Boyadzhiev K.N.: Euler sums of hyperdifferenceharmonic
numbers. J. Number Theory \textbf{147,} 490--498 (2015)

\bibitem {DilMun}Dil A., Muniro\u{g}lu E.: Applications of derivative and
difference operators on some sequences. arXiv:1910.01876.

\bibitem {E}Euler L.: Institutiones Calculi Differentialis Vol II. St
Petersburg, Russia: Academie Imperiale des Sciences, (in Latin) (1755)

\bibitem {F}Flajolet P., Salvy B.: Euler sums and contour integral
representations. Experiment. Math., \textbf{7}(1), 15--35 (1998)

\bibitem {G}Graham R.L., Knuth D.E., Patashnik O.: Concrete Mathematics.
Addison-Wesley Publ. Com. (1994)

\bibitem {Guo}Guo B., Qi F.: Some integral representations and properties of
Lah numbers. J. Algebra and Number Theory Academia \textbf{4}(3), 77--87 (2014)

\bibitem {GandQ1}Guo B., Qi F.: On the sum of the Lah numbers and zeros of the
Kummer confluent hypergeometric function. Acta Univ. Sapientiae, Mathematica,
\textbf{10}(1), 125--133 (2018)

\bibitem {Goyal}Goyal S.P., Laddha R.K.: On the generalized Riemann zeta
functions and the generalized Lambert transform. Ganita Sandesh \textbf{11,}
99--108 (1997)

\bibitem {Kamano}Kamano K.: Dirichlet series associated with hyperharmonic
numbers. Mem. Osaka Inst. Tech. Ser. A \textbf{56}(2), 11--15 (2011)

\bibitem {Kargin}Kargin L., Corcino R.B.: Generalization of Mellin derivative
and its applications. Integral Transforms Spec Funct \textbf{27}, 620--631 (2016)

\bibitem {Kargin2}Kargin L.: Some formulae for products of geometric
polynomials with applications. J. Integer Seq. \textbf{20,} Article 17.4.4. (2017)

\bibitem {Kargin1}Kargin L., \c{C}ekim B.: Higher order generalized geometric
polynomials, Turk. J. Math. \textbf{42}, 887--903 (2018)

\bibitem {Keller}Kellner B.C.: Identities between polynomials related to
Stirling and harmonic numbers, Integers \textbf{14}, \#A54 (2014)

\bibitem {Knopf}Knopf P.M.: The operator $\left(  x\frac{d}{dx}\right)  ^{n}$
and its application to series. Math. Mag. \textbf{76}, 364--371 (2003)

\bibitem {Mezo-Dil}Mez\H{o} I., Dil A.: Euler-Seidel method for certain
combinatorial numbers and a new characterization of Fibonacci sequence. Cent.
Eur. J. Math. \textbf{7}(2), 310--321 (2009)

\bibitem {MandR}Mez\H{o} I., Ramirez J.L.: Some identities for the $r$-Whitney
numbers. Aequationes Math. \textbf{90}, 393--406 (2016)

\bibitem {MR}Mihoubi M., Rahmani M.: The partial r-Bell polynomials. arXiv:1308.0863.

\bibitem {NyRa}Nyul G., R\'{a}cz G.: The $r$-Lah numbers. Discrete Math.
\textbf{338}(10), 1660--1666 (2015)

\bibitem {R}Rao R.S.R.C., Sarma A.S.R.: Some identities involving the Riemann
zeta function. Indian J. Pure Appl. Math. \textbf{10}, 602--607 (1979)

\bibitem {S}S\'{a}ndor J., Crstici B.: Handbook of Number Theory Vol II.
Kluwer Academic Publishers, London (2004)
\end{thebibliography}
\end{document}